\documentclass[twocolumn]{autart}    

\usepackage{amsmath} 
\usepackage{amssymb}  

\usepackage[T1]{fontenc}

\newenvironment{proof}{\emph{Proof.}}{\hfill $\square$ \\}

\newtheorem{algo}[thm]{Algorithm}

\renewcommand\labelenumi{(\roman{enumi})}
\renewcommand\theenumi\labelenumi

\begin{document}

\begin{frontmatter}

\title{Duality of Geometric Tests for Forward-Flatness} 

\author{Johannes Schrotshamer\corauthref{cor}}\ead{johannes.schrotshamer@jku.at},    
\author{Bernd Kolar}\ead{bernd.kolar@jku.at},  
\author{Markus Sch{\"o}berl}\ead{markus.schoeberl@jku.at}              

\corauth[cor]{Corresponding author.}

\address{Institute of Automatic Control and Control Systems Technology, Johannes Kepler University, Linz, Austria}


\begin{abstract}                          
Recently it has been shown that the property of forward-flatness for discrete-time systems, which is a generalization of static feedback linearizability and a special case of a more general concept of flatness, can be checked by two different geometric tests. One is based on unique sequences of involutive distributions, while the other is based on a unique sequence of integrable codistributions. In this paper, the relation between these sequences is discussed and it is shown that the tests are in fact dual. The presented results are illustrated by an academic example.
\end{abstract}

\end{frontmatter}

\section{Introduction}

The concept of flatness has been introduced by Fliess, L{\'e}vine,
Martin and Rouchon in the 1990s for nonlinear continuous-time systems,
see e.g. \cite{FliessLevineMartinRouchon:1992} or \cite{FliessLevineMartinRouchon:1995}.
Flat continuous-time systems possess the characteristic property that
all system variables can be expressed by a flat output and its time
derivatives. This property facilitates elegant solutions for trajectory
planning and tracking problems, making the concept of flatness highly
relevant in practical applications. However, currently no verifiable
necessary and sufficient conditions to check the flatness of nonlinear
multi-input systems do exist (see e.g. \cite{NicolauRespondek:2016},
\cite{NicolauRespondek:2017}, or \cite{GstottnerKolarSchoberl:2021b}).

For discrete-time systems, replacing time derivatives of the continuous-time
definition by forward-shifts leads to the concept of forward-flatness
as discussed in \cite{KaldmaeKotta:2013} or \cite{KolarKaldmaeSchoberlKottaSchlacher:2016}.
This is a special case of the more general approaches described in
\cite{DiwoldKolarSchoberl:2020} or \cite{GuillotMillerioux2020}.
In \cite{KolarDiwoldSchoberl:2023} it has been shown that forward-flatness
can be checked by computing a unique sequence of involutive distributions.
Thus, the forward-flatness of discrete-time systems can already be
checked systematically, in contrast to the continuous-time case. Furthermore,
in the recent paper \cite{kolar2024dual}, another test for forward-flatness
has been derived, which is based on a unique sequence of integrable
codistributions. In the present contribution, we show that the sequences
annihilate each other, proving that the tests are dual.\\
It has also been shown through different approaches in \cite{KolarDiwoldSchoberl:2023}
and \cite{kolar2024dual} that if the system is forward-flat, repeated
system decompositions into a subsystem and an endogenous dynamic feedback
are possible by using suitable state- and input transformations. The
method proposed in \cite{KolarDiwoldSchoberl:2023} requires straightening
out distributions to derive the necessary input transformation. In
contrast, the method from \cite{kolar2024dual} only requires normalizing
a certain number of equations. We will show that the input transformation
derived from the method in \cite{kolar2024dual} also straightens
out the corresponding distribution, making it a special case of the
possible input transformations derived from the method in \cite{KolarDiwoldSchoberl:2023}.

The paper is organized as follows: In Section \ref{sec:Discrete-Time-Systems},
we first recapitulate the concept of forward-flatness for discrete-time
systems, followed by an overview of the existing geometric tests for
forward-flatness from \cite{KolarDiwoldSchoberl:2023} and \cite{kolar2024dual}
in Section \ref{sec:Tests forward-flatness}. Then, in Section \ref{sec:Duality},
we prove that the two proposed tests are dual. In Section \ref{sec:decomp}
we show that the method proposed in \cite{kolar2024dual} for a repeated
system decomposition is a special case of the method proposed in \cite{KolarDiwoldSchoberl:2023}.
Finally, we illustrate our main results based on an academic example
in Section \ref{sec:example}.

\section{Discrete-time Systems and Forward-flatness\label{sec:Discrete-Time-Systems}}

We consider nonlinear discrete-time systems of the form
\begin{equation}
	x^{i,+}=f^{i}(x,u),\hphantom{aa}i=1,\dots,n\label{eq:sysEq}
\end{equation}
with $\dim(x)=n$, $\dim(u)=m$, and smooth functions $f^{i}(x,u)$
that satisfy the submersivity condition
\begin{equation}
	\mathrm{rank}(\partial_{(x,u)}f)=n\,.\label{eq:submersivity}
\end{equation}
Since the submersivity condition \eqref{eq:submersivity} is necessary
for accessibility (see e.g. \cite{Grizzle:1993}), it is not a restriction
and is a common assumption in the discrete-time literature. \\
The map $f$ of \eqref{eq:sysEq} can be interpreted geometrically
as a map from a manifold $\mathcal{X}\times\mathcal{U}$ with coordinates
$(x,u)$ to a manifold $\mathcal{X}^{+}$ with coordinates $x^{+}$,
i.e. 
\begin{equation}
	f:\mathcal{X}\times\mathcal{U}\rightarrow\mathcal{X}^{+}\,.\label{eq:map_f}
\end{equation}
Additionally, for the test from \cite{KolarDiwoldSchoberl:2023} the
map $\pi:\mathcal{X}\times\mathcal{U}\rightarrow\mathcal{X}^{+}$
is needed, which is defined by
\begin{equation}
	x^{i,+}=x^{i},\hphantom{aa}i=1,\dots,n\,.\label{eq:piEq}
\end{equation}
Since the proof of duality in Section \ref{sec:Duality} relies heavily
on the use of adapted coordinates, the coordinates
\begin{equation}
	\begin{aligned}\theta^{i} & =f^{i}(x,u),\hphantom{aa}i=1,\dots,n\\
		\xi^{j} & =h^{j}(x,u),\hphantom{aa}j=1,\dots,m
	\end{aligned}
	\label{eq:adapted_coord}
\end{equation}
on $\mathcal{X}\times\mathcal{U}$ are introduced, where $h^{j}(x,u)$
must be chosen so that the Jacobian matrix of the right-hand side
of \eqref{eq:adapted_coord} is regular. Due to the submersivity property
\eqref{eq:submersivity} this is always possible.\\
The notation with a superscript + is used to denote the forward-shift
of the associated variable, while higher-order forward-shifts are
represented using subscripts in brackets, e.g. $u_{[\alpha]}$ denotes
the $\alpha$th forward-shift of $u$. The notation of 1-forms and
related differential-geometric concepts is adopted from \cite{kolar2024dual}.
To keep formulas short and readable, the Einstein summation convention
is used. We also want to highlight that the findings in \cite{KolarDiwoldSchoberl:2023}
and \cite{kolar2024dual}, along with those presented in this paper,
rely on the inverse function theorem, the implicit function theorem,
and the Frobenius theorem. These theorems inherently yield local results,
which is why all our findings are only valid locally in a suitable
neighborhood of an equilibrium $(x_{0},u_{0})$ of \eqref{eq:sysEq}.\\
In the following, we recapitulate the differential-geometric concepts
of projectable distributions and invariant codistributions, which
form the mathematical foundation upon which the contributions \cite{KolarDiwoldSchoberl:2023}
and \cite{kolar2024dual} are based. Checking whether a certain vector
field is projectable or not becomes very simple in adapted coordinates.
A general vector field 
\[
v=a^{i}(\theta,\xi)\partial_{\theta^{i}}+b^{j}(\theta,\xi)\partial_{\xi^{j}}
\]
on $\mathcal{X}\times\mathcal{U}$ is called ``projectable'' (w.r.t.
the map \eqref{eq:map_f}) if and only if the functions $a^{i}$ are
independent of the coordinates $\xi$, i.e., if the vector field meets
\begin{equation}
	v=a^{i}(\theta)\partial_{\theta^{i}}+b^{j}(\theta,\xi)\partial_{\xi^{j}}\,.\label{eq:proj_vctor_field}
\end{equation}
The pointwise pushforward $f_{*}(v)$ then induces a well-defined
vector field $f_{*}(v)=a^{i}(x^{+})\partial_{x^{i,+}}$ on $\mathcal{X}^{+}$.
Furthermore, we call a distribution $D$ on $\mathcal{X}\times\mathcal{U}$
projectable (w.r.t. $f$), if it admits a basis that consists of projectable
vector fields \eqref{eq:proj_vctor_field} only.\\
A $p$-dimensional codistribution $P=\mathrm{span}\{\omega^{1},\ldots,\omega^{p}\}$,
defined on an $n$-dimensional manifold $\mathcal{M}$ with local
coordinates $(x^{1},\ldots,x^{n})$, is called invariant w.r.t. a
vector field $v$ if the Lie-derivative $L_{v}\omega$ meets $L_{v}\omega\in P$
for all 1-forms $\omega\in P$. Accordingly, $P$ is called invariant
w.r.t. a $d$-dimensional distribution $D=\mathrm{span}\{v_{1},\ldots,v_{d}\}$
if $L_{v}\omega\in P$ for all 1-forms $\omega\in P$ and vector fields
$v\in D$.

To summarize the concept of forward-flatness, we introduce a space
with coordinates $(x,u,u_{[1]},u_{[2]},\ldots)$. Future values of
a smooth function $g$ defined on this space can be determined by
repeatedly applying the forward-shift operator $\delta$, which is
defined according to the rule
\[
\delta(g(x,u,u_{[1]},u_{[2]},\ldots))=g(f(x,u),u_{[1]},u_{[2]},u_{[3]},\ldots)\,.
\]
In this contribution, we will only need backward-shifts of functions
of the form $g(f(x,u))$. For this case, the inverse operator $\delta^{-1}$
is defined as $\delta^{-1}(g(f(x,u)))=g(x)$. Additionally, the dual
test outlined in \cite{kolar2024dual} necessitates backward-shifts
of codistributions spanned by 1-forms of the form $\omega_{i}(f(x,u))\mathrm{d}f^{i}$.
As the shift operators are applied to 1-forms by shifting both their
coefficients and differentials, the backward-shift results in 
\[
\delta^{-1}(\omega_{i}(f(x,u))\mathrm{d}f^{i})=\omega_{i}(x)\mathrm{d}x^{i}
\]
and the backward-shift of codistributions with base elements of this
form is defined accordingly. Now the concept of forward-flatness can
be defined as follows.
\begin{defn}
	\label{def:forward-flatness}The system \eqref{eq:sysEq} is said
	to be forward-flat around an equilibrium $(x_{0},u_{0})$, if the
	$n+m$ coordinate functions $x$ and $u$ can be expressed locally
	by an $m$-tuple of functions
	\begin{equation}
		y^{j}=\varphi^{j}(x,u,u_{[1]},\ldots,u_{[q]})\,,\quad j=1,\ldots,m\label{eq:flat_output}
	\end{equation}
	and their forward-shifts $y_{[1]}=\delta(\varphi(x,u,u_{[1]},\ldots,u_{[q]}))$,
	$y_{[2]}=\delta^{2}(\varphi(x,u,u_{[1]},\ldots,u_{[q]}))$, $\ldots$
	up to some finite order. The $m$-tuple \eqref{eq:flat_output} is
	called a flat output.
\end{defn}

The representation of $x$ and $u$ by the flat output and its forward-shifts
is unique and has the form\footnote{The multi-index $R=(r_{1},\ldots,r_{m})$ contains the number of forward-shifts
	of the individual components of the flat output which is needed to
	express $x$ and $u$, and $y_{[0,R]}$ is an abbreviation for $y$
	and its forward-shifts up to order $R$.}
\begin{equation}
	x^{i}=F_{x}^{i}(y_{[0,R-1]})\,,\:u^{j}=F_{u}^{j}(y_{[0,R]})\label{eq:flat_param}
\end{equation}
with $i=1,\ldots,n$ and $j=1,\ldots,m$. The term forward-flatness
implies that both the flat output \eqref{eq:flat_output} and the
corresponding parameterization of the system variables \eqref{eq:flat_param}
only involve forward-shifts. In contrast, the more general case discussed
in \cite{DiwoldKolarSchoberl:2020} or \cite{GuillotMillerioux2020}
also involves backward-shifts. As demonstrated in \cite{KolarSchoberlDiwold:2019},
forward-flat systems can always be transformed into a special triangular
form.
\begin{defn}
	\label{def:basic_decomposition_flat}A forward-flat system (\ref{eq:sysEq})
	can be transformed by a state- and input transformation\begin{subequations}\label{eq:decomposition_coord_transformation}
		\begin{align}
			(\bar{x}_{1},\bar{x}_{2})=\:\: & \Phi_{x}(x)\label{eq:decompostion_state_transformation}\\
			(\bar{u}_{1},\bar{u}_{2})=\:\: & \Phi_{u}(x,u)\label{eq:decompostion_input_transformation}
		\end{align}
	\end{subequations}into the form\begin{subequations}
		\label{eq:basic_decomposition_flat}
		\begin{align}
			\bar{x}_{2}^{+} & =f_{2}(\bar{x}_{2},\bar{x}_{1},\bar{u}_{2})\label{eq:decomposition_subsystem}\\
			\bar{x}_{1}^{+} & =f_{1}(\bar{x}_{2},\bar{x}_{1},\bar{u}_{2},\bar{u}_{1})\label{eq:decomposition_feedback}
		\end{align}
	\end{subequations} with $\dim(\bar{x}_{1})\geq1$ and $\mathrm{rank}(\partial_{\bar{u}_{1}}f_{1})=\dim(\bar{x}_{1})$. 
\end{defn}

According to Lemma 3 from \cite{kolar2024dual}, the importance of
the triangular form (\ref{eq:basic_decomposition_flat}) lies in the
fact that a system of the form (\ref{eq:basic_decomposition_flat})
is forward-flat if and only if the subsystem \eqref{eq:decomposition_subsystem}
is forward-flat. The equations \eqref{eq:decomposition_feedback}
are just an endogenous dynamic feedback for the subsystem \eqref{eq:decomposition_subsystem}.
As a result, the flat output of (\ref{eq:basic_decomposition_flat})
can be systematically determined by extending the flat output of \eqref{eq:decomposition_subsystem}
with the redundant inputs of (\ref{eq:basic_decomposition_flat})
like in Lemma 3 of \cite{kolar2024dual}.

\section{Geometric Tests for Forward-flatness\label{sec:Tests forward-flatness}}

In this section, a short overview of the geometric tests for forward-flatness
of \cite{KolarDiwoldSchoberl:2023} and \cite{kolar2024dual} is given.

\subsection{Distribution-based Test for Forward-flatness\label{subsec:Geometric}}

As shown in \cite{KolarDiwoldSchoberl:2023}, it is possible to formulate
a test for forward-flatness based on certain sequences of distributions.
For a better comparison of the two tests, Algorithm 1 from \cite{KolarDiwoldSchoberl:2023}
is rewritten with a focus on the sequence of involutive distributions
$E_{0}\subset E_{1}\subset\cdots\subset E_{\bar{k}-1}$.

\begin{algo}\label{alg:definition_sequence}\label{alg:distribution}Start
	with the distribution $E_{0}=\mathrm{span}\{\partial_{u}\}$ on $\mathcal{X}\times\mathcal{U}$
	and repeat the following steps for $k\geq1$:\textbf{\vspace{-0.2cm}
	}
	\begin{enumerate}
		\item[\emph{1.}] Determine the largest projectable subdistribution $D_{k-1}$ that
		meets $D_{k-1}\subseteq E_{k-1}$.
		\item[\emph{2.}] Compute the pushforward $\Delta_{k}=f_{*}(D_{k-1})$.
		\item[\emph{3.}] Proceed with $E_{k}=\pi_{*}^{-1}(\Delta_{k})$.\textbf{\vspace{-0.1cm}
			\vspace{-0.1cm}
		}
	\end{enumerate}
	Stop if $\dim(E_{\bar{k}})=\dim(E_{\bar{k}-1})$ for some step $\bar{k}$.
	
\end{algo}
\begin{thm}
	\label{thm:condition}The system \eqref{eq:sysEq} is forward-flat,
	if and only if $\dim(E_{\bar{k}-1})=n+m$.
\end{thm}

On $\mathcal{X}\times\mathcal{U}$, Algorithm \ref{alg:distribution}
yields a unique nested sequence of projectable and involutive distributions
\begin{equation}
	D_{0}\subset D_{1}\subset\cdots\subset D_{\bar{k}-2}\label{eq:sequ_D_k}
\end{equation}
and a unique nested sequence of involutive distributions 
\begin{equation}
	\mathrm{span}\{0\}=\Delta_{0}\subset\Delta_{1}\subset\cdots\subset\Delta_{\bar{k}-1}\label{eq:sequ_Delta_k}
\end{equation}
on $\mathcal{X}^{+}$, which are related by $\Delta_{k}=f_{*}(D_{k-1})$.
With $E_{k}=\pi_{*}^{-1}(\Delta_{k})$ and due to the fact that the
involutivity of $\Delta_{k}$ implies the involutivity of $E_{k}$,
it also yields a unique nested sequence of involutive distributions
\begin{equation}
	\mathrm{span}\{\partial_{u}\}=E_{0}\subset E_{1}\subset\cdots\subset E_{\bar{k}-1}\:.\label{eq:E_sequ}
\end{equation}
If a system \eqref{eq:sysEq} is forward-flat, then the corresponding
distributions $\Delta_{\bar{k}-1}$ and $E_{\bar{k}-1}$ result in
$\Delta_{\bar{k}-1}=\mathrm{span}\{\partial_{x^{+}}\}$ and $E_{\bar{k}-1}=\mathrm{span}\{\partial_{x},\partial_{u}\}$.

\subsection{Codistribution-based Test for Forward-flatness\label{subsec:Geometric-Dual}}

It is also possible to formulate a test to assess the forward-flatness
of a system \eqref{eq:sysEq} based on a certain sequence of integrable
codistributions, as proposed in \cite{kolar2024dual}.

\begin{algo}\label{alg:definition_sequence}\label{alg:dual}Start
	with the codistribution $P_{1}=\mathrm{span}\{\mathrm{d}x\}$ on $\mathcal{X}\times\mathcal{U}$
	and repeat the following steps for $k\geq1$:\textbf{\vspace{-0.2cm}
	}
	\begin{enumerate}
		\item[\emph{1.}] Compute the intersection $P_{k}\cap\mathrm{span}\{\mathrm{d}f\}$.
		\item[\emph{2.}] Determine the smallest codistribution $P_{k+1}^{+}$ which is invariant
		w.r.t. the distribution $\mathrm{span}\{\mathrm{d}f\}^{\perp}$ and
		contains $P_{k}\cap\mathrm{span}\{\mathrm{d}f\}$.
		\item[\emph{3.}] Define $P_{k+1}=\delta^{-1}(P_{k+1}^{+})$.
	\end{enumerate}
	\textbf{\vspace{-0.1cm}
		\vspace{-0.1cm}
	}Stop if $P_{\bar{k}+1}=P_{\bar{k}}$ for some step $\bar{k}$.
	
\end{algo}

On $\mathcal{X}\times\mathcal{U}$, Algorithm \ref{alg:dual} yields
a unique nested sequence of integrable codistributions 
\begin{equation}
	P_{\bar{k}}\subset P_{\bar{k}-1}\subset\ldots\subset P_{1}=\mathrm{span}\{\mathrm{d}x\}\,.\label{eq:P_sequ}
\end{equation}

\begin{thm}
	\label{thm:dual_condition}The system \eqref{eq:sysEq} is forward-flat,
	if and only if $P_{\bar{k}}=0$.
\end{thm}

\section{Duality of the Sequences\label{sec:Duality}}

Before we prove our main result -- the duality of the sequences \eqref{eq:E_sequ}
and \eqref{eq:P_sequ} defined by Algorithm \ref{alg:distribution}
and \ref{alg:dual} -- in Subsection \ref{subsec:Duality}, we derive
some important technical results.

\subsection{Technical Results}

The following results apply to a general $d$-dimensional distribution
$D$ on the $n+m$-dimensional manifold $\mathcal{X}\times\mathcal{U}$
and its annihilator $P=D^{\perp}$.\\
Throughout, we use adapted coordinates \eqref{eq:adapted_coord}.
First, up to a renumbering of coordinates, $D$ admits a basis with
normalized vector fields 
\begin{equation}
	\begin{aligned}v_{k} & =\partial_{\theta^{k}}+a_{k}^{\bar{d}+1}(\theta,\xi)\partial_{\theta^{\bar{d}+1}}+\cdots+a_{k}^{n}(\theta,\xi)\partial_{\theta^{n}}\\
		& +b_{k}^{1}(\theta,\xi)\partial_{\xi^{1}}+\cdots+b_{k}^{m-d+\bar{d}}(\theta,\xi)\partial_{\xi^{m-d+\bar{d}}}
	\end{aligned}
	\label{eq:v_bar_d_norm}
\end{equation}
\begin{equation}
	v_{l}=b_{l}^{1}(\theta,\xi)\partial_{\xi^{1}}+\cdots+b_{l}^{m-d+\bar{d}}(\theta,\xi)\partial_{\xi^{m-d+\bar{d}}}+\partial_{\xi^{m-d+l}}\label{eq:v_d_norm}
\end{equation}
with $k=1,\ldots,\bar{d}$ and $l=\bar{d}+1,\ldots,d$. A straightforward
calculation shows that its $(n+m-d)$-dimensional annihilator $P$
is spanned by the 1-forms 
\begin{equation}
	\omega^{s}=\alpha_{1}^{s}(\theta,\xi)\mathrm{d}\theta^{1}+\cdots+\alpha_{\bar{d}}^{s}(\theta,\xi)\mathrm{d}\theta^{\bar{d}}+\mathrm{d}\theta^{\bar{d}+s}\label{eq:omega_theta_norm}
\end{equation}
\begin{equation}
	\begin{aligned}\omega^{t} & =\alpha_{1}^{t}(\theta,\xi)\mathrm{d}\theta^{1}+\cdots+\alpha_{\bar{d}}^{t}(\theta,\xi)\mathrm{d}\theta^{\bar{d}}+\mathrm{d}\xi^{t-n+\bar{d}}\\
		& +\beta_{m-d+\bar{d}+1}^{t}(\theta,\xi)\mathrm{d}\xi^{m-d+\bar{d}+1}+\cdots+\beta_{m}^{t}(\theta,\xi)\mathrm{d}\xi^{m}
	\end{aligned}
	\label{eq:omega_xi_norm}
\end{equation}
with $s=1,\ldots,n-\bar{d}$ and $t=n-\bar{d}+1,\ldots,n+m-d$\footnote{So that there occurs a $(n+m-d)\times(n+m-d)$ identity matrix in
	the coefficients of the coordinates $\theta^{\bar{d}+1},\ldots,\theta^{n},\xi^{1},\ldots,\xi^{m-d+\bar{d}}$.}, with the coefficients given by
\begin{equation}
	\alpha_{k}^{s}=-a_{k}^{s+\bar{d}},\,\alpha_{k}^{t}=-b_{k}^{t-n+\bar{d}}\text{ and }\beta_{l}^{t}=-b_{l}^{t-n+\bar{d}}\:.\label{eq:coeffs}
\end{equation}
Next, we reformulate the computation of the largest projectable subdistribution
of the distribution $D$, such that in contrast to the multi-step
procedure presented in \cite{KolarDiwoldSchoberl:2023}, the computations
are summarized in one single step. This will be essential in order
to prove our main result in Subsection \ref{subsec:Duality}.
\begin{prop}
	\label{prop:D_bar}Consider a $d$-dimensional distribution $D$ with
	normalized base elements of the form \eqref{eq:v_bar_d_norm}, \eqref{eq:v_d_norm}
	and the $(n-\bar{d})\times\bar{d}$-matrix $L=\left[\begin{array}{lll}
		a_{1}^{\bar{i}}(\theta,\xi) & \cdots & a_{\bar{d}}^{\bar{i}}(\theta,\xi)\end{array}\right]$ with $\bar{i}=\bar{d}+1,\ldots,n$. Combining all partial derivatives
	of $L$ w.r.t. the coordinates $\xi$ in one matrix, results in 
	\begin{equation}
		M=\begin{bmatrix}\partial_{\xi^{j}}L\\
			\partial_{\xi^{p}}\partial_{\xi^{j}}L\\
			\vdots
		\end{bmatrix}\label{eq:matrix_M}
	\end{equation}
	with\textup{ $j,p=1,\ldots,m$}\footnote{The matrix $M$ is composed of all (also higher order, repeated and
		mixed) partial derivatives of $L$, until no further linearly independent
		rows are obtained. The matrix $M$ certainly contains a finite number
		of linearly independent rows, since the rank is limited by the number
		of columns $\bar{d}$.}\textup{. W}ith the column vectors \textup{$\left[\begin{array}{lll}
			c_{1}^{1}(\theta) & \cdots & c_{1}^{\bar{d}}(\theta)\end{array}\right]^{T},\ldots,$ $\left[\begin{array}{lll}
			c_{q}^{1}(\theta) & \cdots & c_{q}^{\bar{d}}(\theta)\end{array}\right]^{T}$ }forming a basis for the\textup{ $q$}-dimensional kernel of the
	matrix $M$, the largest projectable subdistribution of $D$ is given
	by
	\begin{equation}
		\begin{aligned}\bar{D} & =\mathrm{span}\left\{ c_{1}^{1}(\theta)v_{1}+\cdots+c_{1}^{\bar{d}}(\theta)v_{\bar{d}},\ldots,\right.\\
			& \left.c_{q}^{1}(\theta)v_{1}+\cdots+c_{q}^{\bar{d}}(\theta)v_{\bar{d}},v_{\bar{d}+1},\ldots,v_{d}\right\} \\
			& =\mathrm{span}\{\bar{v}_{1},\ldots,\bar{v}_{q},v_{\bar{d}+1},\ldots,v_{d}\}\,.
		\end{aligned}
		\label{eq:basis_projectable_subdist}
	\end{equation}
\end{prop}

\begin{proof}
	A linear combination $c^{1}(\theta,\xi)v_{1}+\cdots+c^{\bar{d}}(\theta,\xi)v_{\bar{d}}$
	of the $\bar{d}$ base vector fields \eqref{eq:v_bar_d_norm} is projectable
	if and only if the resulting vector field is of the form \eqref{eq:proj_vctor_field}
	(the $d-\bar{d}$ base elements $v_{\bar{d}+1},\ldots,v_{d}$ in the
	form \eqref{eq:v_d_norm} are already projectable, since they project
	to zero). Because of the $\bar{d}\times\bar{d}$ identity matrix in
	the coefficients of the base elements \eqref{eq:v_bar_d_norm}, this
	can only be achieved with coefficients $c^{1}(\theta),\ldots,c^{\bar{d}}(\theta)$
	that are independent of the variables $\xi$. In addition, with the
	column vector of coefficients
	\begin{equation}
		c(\theta)=\left[\begin{array}{lll}
			c^{1}(\theta) & \cdots & c^{\bar{d}}(\theta)\end{array}\right]^{T}\,,\label{eq:coeffs_linearcomb}
	\end{equation}
	the remaining $n-\bar{d}$ coefficients 
	\begin{equation}
		\hat{a}^{\bar{i}}=\left[\begin{array}{lll}
			a_{1}^{\bar{i}}(\theta,\xi) & \cdots & a_{\bar{d}}^{\bar{i}}(\theta,\xi)\end{array}\right]c(\theta)\label{eq:a_hat}
	\end{equation}
	with $\bar{i}=\bar{d}+1,\ldots,n$ are independent of the coordinates
	$\xi$ if and only if all (also higher order, repeated and mixed)
	partial derivatives w.r.t. $\xi$ are equal to $0$, i.e. $\partial_{\xi^{j}}\hat{a}^{\bar{i}}=0$,
	$\partial_{\xi^{p}}\partial_{\xi^{j}}\hat{a}^{\bar{i}}=0,\ldots$
	with $j,p=1,\ldots,m$. Combining these conditions with equation \eqref{eq:a_hat}
	results in
	\begin{equation}
		\begin{aligned}\partial_{\xi^{j}}\left(\left[\begin{array}{lll}
				a_{1}^{\bar{i}}(\theta,\xi) & \cdots & a_{\bar{d}}^{\bar{i}}(\theta,\xi)\end{array}\right]\right)c(\theta) & =0\\
			\partial_{\xi^{p}}\partial_{\xi^{j}}\left(\left[\begin{array}{lll}
				a_{1}^{\bar{i}}(\theta,\xi) & \cdots & a_{\bar{d}}^{\bar{i}}(\theta,\xi)\end{array}\right]\right)c(\theta) & =0\\
			& \enskip\vdots
		\end{aligned}
		\label{eq:conditions}
	\end{equation}
	as the vector $c(\theta)$ is independent of $\xi$. When combining
	all partial derivatives of $\left[\begin{array}{lll}
		a_{1}^{\bar{i}}(\theta,\xi) & \cdots & a_{\bar{d}}^{\bar{i}}(\theta,\xi)\end{array}\right]$ from \eqref{eq:conditions} in one matrix, the conditions for the
	coefficients $\hat{a}^{\bar{i}}$ to be independent of $\xi$ are
	satisfied if and only if \eqref{eq:coeffs_linearcomb} lies in the
	kernel of the matrix \eqref{eq:matrix_M}. \\
	Since the matrix consists of a maximal number of linearly independent
	partial derivatives of the rows of $\left[\begin{array}{lll}
		a_{1}^{\bar{i}}(\theta,\xi) & \cdots & a_{\bar{d}}^{\bar{i}}(\theta,\xi)\end{array}\right]$ w.r.t. $\xi$, the $\xi$-dependence of $M$ can be eliminated by
	row manipulations. Hence, $\mathrm{ker}(M)$ is independent of $\xi$,
	and a basis for $\mathrm{ker}(M)$ forms a maximal set of independent
	coefficient-vectors \eqref{eq:coeffs_linearcomb}.
\end{proof}
Based on Proposition \ref{prop:D_bar}, the dimension of the largest
projectable subdistribution of $D$ can be easily derived as follows.
\begin{cor}
	\label{cor:dim_Dq}The dimension of the largest projectable subdistribution
	$\bar{D}$ of a given distribution $D$ is determined by 
	\begin{equation}
		\mathrm{dim}(\bar{D})=\mathrm{dim}(D)-\mathrm{rank}(M)\label{eq:dim_Dq}
	\end{equation}
	with the matrix $M$ according to \eqref{eq:matrix_M}.
\end{cor}

\begin{proof}
	The result immediately follows from $\mathrm{dim}(\bar{D})=d-\bar{d}+\mathrm{dim}(\mathrm{ker}(M))$
	combined with $\mathrm{dim}(\mathrm{im}(M))+\mathrm{dim}(\mathrm{ker}(M))=\bar{d}$.
\end{proof}
Establishing a connection between the largest projectable subdistribution
$\bar{D}\subset D$ and the smallest $\xi$-invariant extension of
the codistribution $P\cap\mathrm{span}\{\mathrm{d}\theta\}$ with
$P=D^{\perp}$, is essential for proving our main result in Theorem
\ref{thm:main_thm}.
\begin{prop}
	\label{prop:P_hat}Consider a $d$-dimensional distribution $D$ and
	its $(n+m-d)$-dimensional annihilator $P$ on an $(n+m)$-dimensional
	manifold $\mathcal{\mathcal{X\times U}}$ with normalized base elements
	\eqref{eq:v_bar_d_norm}-\eqref{eq:omega_xi_norm}. The smallest codistribution
	$\hat{P}$, which is invariant w.r.t. the distribution $\mathrm{span}\{\partial_{\xi}\}$
	and contains $P\cap\mathrm{span}\{\mathrm{d}\theta\}=\mathrm{span}\{\omega^{1},\ldots,\omega^{n-\bar{d}}\}$,
	is given by $\hat{P}=\mathrm{span}\{\rho^{1},\ldots,\rho^{\mathrm{rank}(M)},\omega^{1},\ldots,\omega^{n-\bar{d}}\}$
	with the 1-forms 
	\begin{equation}
		\begin{bmatrix}\rho^{1}\\
			\vdots\\
			\rho^{\mathrm{rank}(M)}
		\end{bmatrix}=-\hat{M}\begin{bmatrix}\mathrm{d}\theta^{1}\\
			\vdots\\
			\mathrm{d}\theta^{\bar{d}}
		\end{bmatrix}\label{eq:added_1forms}
	\end{equation}
	and the matrix $\hat{M}$, which consists of the linearly independent
	rows of \eqref{eq:matrix_M}.
\end{prop}

\begin{proof}
	Since $P$ is given with normalized base elements \eqref{eq:omega_theta_norm}
	and \eqref{eq:omega_xi_norm}, the intersection $P\cap\mathrm{span}\{\mathrm{d}\theta\}$
	is obviously spanned by the $n-\bar{d}$ 1-forms \eqref{eq:omega_theta_norm}.
	By extending $P\cap\mathrm{span}\{\mathrm{d}\theta\}=\mathrm{span}\{\omega^{1},\ldots,\omega^{n-\bar{d}}\}$
	with suitable Lie derivatives $\rho$ of the 1-forms \eqref{eq:omega_theta_norm}
	as discussed in the proof of Proposition 6 in \cite{kolar2024dual},
	a basis for the smallest codistribution $\hat{P}$ which is invariant
	w.r.t. the distribution $\mathrm{span}\{\partial_{\xi}\}$ and contains
	$P\cap\mathrm{span}\{\mathrm{d}\theta\}$ is obtained. Considering
	that the repeated and mixed Lie derivatives of a 1-form \eqref{eq:omega_theta_norm}
	w.r.t. the vector fields $\partial_{\xi^{j}},\partial_{\xi^{p}},\ldots$
	are due to the normalized form given by 
	\begin{equation}
		\begin{aligned}L_{\partial_{\xi^{j}}}\omega^{s} & =\partial_{\xi^{j}}\alpha_{1}^{s}(\theta,\xi)\mathrm{d}\theta^{1}+\cdots+\partial_{\xi^{j}}\alpha_{\bar{d}}^{s}(\theta,\xi)\mathrm{d}\theta^{\bar{d}}\\
			L_{\partial_{\xi^{p}}}L_{\partial_{\xi^{j}}}\omega^{s} & =\partial_{\xi^{p}}\partial_{\xi^{j}}\alpha_{1}^{s}(\theta,\xi)\mathrm{d}\theta^{1}+\cdots+\partial_{\xi^{p}}\partial_{\xi^{j}}\alpha_{\bar{d}}^{s}(\theta,\xi)\mathrm{d}\theta^{\bar{d}}\\
			& \enskip\vdots
		\end{aligned}
		\label{eq:Lie_partial}
	\end{equation}
	and that according to \eqref{eq:coeffs} the coefficients of \eqref{eq:v_bar_d_norm}
	and \eqref{eq:omega_theta_norm} are related by $\alpha_{k}^{s}=-a_{k}^{s+\bar{d}}$,
	the linearly independent 1-forms of \eqref{eq:Lie_partial} can be
	written in the form \eqref{eq:added_1forms}. As a result, the 1-forms
	$\{\rho^{1},\ldots,\rho^{\mathrm{rank}(M)},\omega^{1},\ldots,\omega^{n-\bar{d}}\}$
	form a possible basis for $\hat{P}$ with $\hat{P}\subset\mathrm{span}\{\mathrm{d}\theta\}$.
\end{proof}

\subsection{Main Result \label{subsec:Duality}}

In the following, we show that the $E$-sequence of Algorithm \eqref{alg:distribution}
and the $P$-sequence of Algorithm \eqref{alg:dual}, which are both
defined on $\mathcal{X}\times\mathcal{U}$, annihilate each other.
\begin{thm}
	\label{thm:main_thm}Consider a discrete-time system \eqref{eq:sysEq}
	with the corresponding sequence of distributions \eqref{eq:E_sequ}
	according to Algorithm \ref{alg:distribution} as well as the sequence
	of codistributions \eqref{eq:P_sequ} according to Algorithm \ref{alg:dual}.
	For $k=1,\ldots,\bar{k}$, the sequences \eqref{eq:E_sequ} and \eqref{eq:P_sequ}
	are related by $P_{k}=E_{k-1}^{\perp}$.
\end{thm}

\begin{proof}
	For $k=1$, it is obvious that $P_{1}=\mathrm{span}\{\mathrm{d}x\}$
	annihilates $E_{0}=\mathrm{span}\{\partial_{u}\}$. For all subsequent
	steps, we will show by induction that if $P_{k}=E_{k-1}^{\perp}$,
	then it follows that $P_{k+1}=E_{k}^{\perp}$. The induction step
	is proven by computing $P_{k+1}$ and $E_{k}$ according to Algorithms
	\ref{alg:distribution} and \ref{alg:dual}, and verifying that $E_{k}\rfloor P_{k+1}=0$
	as well as $\mathrm{dim}(E_{k})+\mathrm{dim}(P_{k+1})=n+m$. The proof
	is divided into two steps. In the first step, assuming $P_{k}=E_{k-1}^{\perp}$,
	we show that the largest projectable subdistribution $D_{k-1}\subset E_{k-1}$
	and the smallest $\xi$-invariant codistribution $P_{k+1}^{+}$ containing
	$P_{k}\cap\mathrm{span}\{\mathrm{d}\theta\}$ satisfy the condition
	$D_{k-1}\rfloor P_{k+1}^{+}=0$. Based on this intermediate result,
	in the second step, we show that $P_{k+1}=E_{k}^{\perp}$ .
	
	\emph{Step 1}: Starting with $E_{k-1}=\mathrm{span}\{v_{1},\ldots,v_{d}\}$
	and $P_{k}=\mathrm{span}\{\omega^{1},\ldots,\omega^{n+m-d}\}$, with
	normalized base elements \eqref{eq:v_bar_d_norm}-\eqref{eq:omega_xi_norm},
	the distribution $D_{k-1}=\mathrm{span}\{\bar{v}_{1},\ldots,\bar{v}_{q},v_{\bar{d}+1},\ldots,v_{d}\}$
	and codistribution $P_{k+1}^{+}=\mathrm{span}\{\rho^{1},\ldots,\rho^{\mathrm{rank}(M)},\omega^{1},\ldots,\omega^{n-\bar{d}}\}$
	are determined using Propositions \ref{prop:D_bar} and \ref{prop:P_hat},
	respectively. Consequently, it can be shown that the interior product
	of all their base elements yields zero.\\
	First, due to the assumption that $P_{k}=E_{k-1}^{\perp}$ and because
	of $D_{k-1}\subset E_{k-1}$, it follows that $D_{k-1}\rfloor P_{k}=0.$
	Consequently, because of $\omega^{1},\ldots,\omega^{n-\bar{d}}\in P_{k}$,
	the interior product of these 1-forms $\omega^{1},\ldots,\omega^{n-\bar{d}}\in P_{k+1}^{+}$
	with the base elements of $D_{k-1}$ is zero. Second, the interior
	product of the 1-forms $\rho^{1},\ldots,\rho^{\mathrm{rank}(M)}\in P_{k+1}^{+}$
	with the base vector fields $v_{\bar{d}+1},\ldots,v_{d}\in D_{k-1}$
	is certainly zero, because of $v_{\bar{d}+1},\ldots,v_{d}\subset\mathrm{span}\{\partial_{\xi}\}$
	and $P_{k+1}^{+}\subset\mathrm{span}\{\mathrm{d}\theta\}$. Additionally,
	since according to Proposition \ref{prop:P_hat} the 1-forms $\rho^{1},\ldots,\rho^{\mathrm{rank}(M)}$
	are given by \eqref{eq:added_1forms}, the interior product with the
	remaining base elements $\bar{v}_{1},\ldots,\bar{v}_{q}\in D_{k-1}$
	can in matrix notation be written as
	\begin{equation}
		-\hat{M}\begin{bmatrix}c_{1}^{1}(\theta) & \cdots & c_{q}^{1}(\theta)\\
			\vdots & \ddots & \vdots\\
			c_{1}^{\bar{d}}(\theta) & \cdots & c_{q}^{\bar{d}}(\theta)
		\end{bmatrix}=\underline{0}\label{eq:M_kern}
	\end{equation}
	due to the $\bar{d}\times\bar{d}$ identity matrix in the coefficients
	of the base elements \eqref{eq:v_bar_d_norm}. As the column vectors
	$\left[\begin{array}{lll}
		c_{1}^{1}(\theta) & \cdots & c_{1}^{\bar{d}}(\theta)\end{array}\right]^{T},\ldots,$ $\left[\begin{array}{lll}
		c_{q}^{1}(\theta) & \cdots & c_{q}^{\bar{d}}(\theta)\end{array}\right]^{T}$ form a basis for the kernel of the matrix $M$ (also for $\hat{M}$),
	equation \eqref{eq:M_kern} certainly holds true and shows that $D_{k-1}\rfloor P_{k+1}^{+}=0$.
	
	\emph{Step 2}: The codistribution $P_{k+1}$ is determined according
	to Step 3 in Algorithm \ref{alg:dual} by a backward-shift of $P_{k+1}^{+}=\mathrm{span}\{\rho^{1},\ldots,\rho^{\mathrm{rank}(M)},\omega^{1},\ldots,$
	$\omega^{n-\bar{d}}\}$. Since the extended codistribution $P_{k+1}^{+}$
	is invariant w.r.t. the distribution $\mathrm{span}\{\partial_{\xi}\}$,
	there exists a basis with 1-forms independent of $\xi^{j}$. Considering
	that $\theta^{i}=f^{i}(x,u)$, the application of the backward-shift
	operator simply replaces the coordinates $\theta$ with the coordinates
	$x$.\\
	Furthermore, according to Steps 2 and 3 of Algorithm \ref{alg:distribution},
	$E_{k}$ is calculated by first computing the pushforward $\Delta_{k}=f_{*}(D_{k-1})$
	on the $n$-dimensional manifold $\mathcal{X}^{+}$. In adapted coordinates
	$(\theta,\xi)$ the pushforward takes the simple form $\Delta_{k}=pr_{1,*}(D_{k-1})$
	with the coordinates $\theta$ being substituted by $x^{+}$. Since
	the vector fields $v_{\bar{d}+1},\ldots,v_{d}\in D_{k-1}$ project
	to zero, it follows that $\mathrm{dim}(\Delta_{k})=\mathrm{dim}(D_{k-1})-(d-\bar{d})$.
	Finally, $E_{k}$ is obtained by the application of the map $E_{k}=\pi_{*}^{-1}(\Delta_{k})$
	from \eqref{eq:piEq}, which replaces the coordinates $x^{+}$ with
	the coordinates $x$ and adds the $m$ vector fields $\partial_{u^{1}},\ldots,\partial_{u^{m}}$.\\
	So, based on $D_{k-1}\rfloor P_{k+1}^{+}=0$, considering that replacing
	coordinates does not affect the result of the interior product and
	that $\partial_{u}\rfloor P_{k+1}=0$ holds due to $P_{k+1}\subset\mathrm{span}\{\mathrm{d}x\}$,
	combined with the construction of $E_{k}$ and $P_{k+1}$ as shown
	above, ensures that the condition $E_{k}\rfloor P_{k+1}=0$ is satisfied.
	Additionally, the dimension of $E_{k}$ follows as 
	\begin{align}
		\mathrm{dim}(E_{k}) & =\mathrm{dim}(D_{k-1})-(d-\bar{d})+m=\nonumber \\
		& =\mathrm{dim}(E_{k-1})-\mathrm{rank}(M)-(d-\bar{d})+m=\label{eq:dim(E_k)}\\
		& =\bar{d}-\mathrm{rank}(M)+m\nonumber 
	\end{align}
	with \eqref{eq:dim_Dq} and $\mathrm{dim}(E_{k-1})=d$. Due to the
	fact that 
	\begin{equation}
		\mathrm{dim}(P_{k+1})=\mathrm{dim}(P_{k+1}^{+})=n-\bar{d}+\mathrm{rank}(M)\,,\label{eq:dim(P_k+1)}
	\end{equation}
	the condition $\mathrm{dim}(E_{k})+\mathrm{dim}(P_{k+1})=n+m$ is
	met, which completes the proof.
\end{proof}
In conclusion, we have proven that starting with the codistribution
$P_{1}=\mathrm{span}\{\mathrm{d}x\}$ and the distribution $E_{0}=\mathrm{span}\{\partial_{u}\}$,
$P_{k}=E_{k-1}^{\perp}$ holds for $k=1,\ldots,\bar{k}$. Thus, the
two tests for forward-flatness of \cite{KolarDiwoldSchoberl:2023}
and \cite{kolar2024dual} can be considered as dual. Furthermore,
in Proposition \ref{prop:D_bar} we established an efficient way to
compute the largest projectable subdistribution $\bar{D}$ of a given
distribution $D$.

\section{Further Properties\label{sec:decomp}}

In Subsection \ref{subsec:Duality}, it has been shown that the sequences
\eqref{eq:E_sequ} and \eqref{eq:P_sequ} annihilate each other. The
relation of the sequence \eqref{eq:sequ_D_k} with the codistributions
of Algorithm \ref{alg:dual} is discussed in the following proposition.
\begin{prop}
	\label{prop:dual_D}Consider a discrete-time system \eqref{eq:sysEq}
	with the corresponding sequence of distributions \eqref{eq:sequ_D_k}
	according to Algorithm \ref{alg:distribution} as well as the sequence
	of codistributions \textup{
		\begin{equation}
			P_{\bar{k}}^{+}+P_{\bar{k}-1}\subset P_{\bar{k}}^{+}+P_{\bar{k}-1}\subset\ldots\subset P_{2}^{+}+P_{1}\label{eq:D_perp}
		\end{equation}
		with $P_{k+1}^{+}$ and $P_{k}$} according to Algorithm \ref{alg:dual}.
	For $k=1,\ldots,\bar{k}-1$, the sequences \eqref{eq:sequ_D_k} and
	\eqref{eq:D_perp} are related by $P_{k+1}^{+}+P_{k}=D_{k-1}^{\perp}$.
\end{prop}

\begin{proof}
	Combining $P_{k+1}^{+}=\mathrm{span}\{\rho^{1},\ldots,\rho^{\mathrm{rank}(M)},\omega^{1},\ldots,$
	$\omega^{n-\bar{d}}\}$, $P_{k}=\mathrm{span}\left\{ \omega^{1},\ldots,\omega^{n+m-d}\right\} $
	and the conditions $D_{k-1}\rfloor P_{k+1}^{+}=0$ and $D_{k-1}\rfloor P_{k}=0$
	from step 1 of the proof of Theorem \ref{thm:main_thm} confirms that
	$D_{k-1}\rfloor(P_{k+1}^{+}+P_{k})=0$ is met. Additionally, because
	of $\mathrm{dim}(P_{k+1}^{+}+P_{k})=n+m-d+\mathrm{rank}(M)$ and \eqref{eq:dim_Dq},
	it follows that $\mathrm{dim}(D_{k-1})+\mathrm{dim}(P_{k+1}^{+}+P_{k})=n+m$,
	which completes the proof.
\end{proof}
If a system \eqref{eq:sysEq} is forward-flat, it can be transformed
into a triangular form (\ref{eq:basic_decomposition_flat}) by a suitable
state- and input transformation \eqref{eq:decompostion_state_transformation}
and \eqref{eq:decompostion_input_transformation} as described in
Definition \ref{def:basic_decomposition_flat}. Such a system decomposition
into a subsystem and an endogenous dynamic feedback can be carried
out in each step $k$ of Algorithms \ref{alg:distribution} and \ref{alg:dual}.
In \cite{KolarDiwoldSchoberl:2023}, these repeated system decompositions
are achieved by straightening out the sequences of involutive distributions
\eqref{eq:sequ_Delta_k} and \eqref{eq:sequ_D_k}. At the first decomposition
step, due to $E_{1}=\pi_{*}^{-1}(\Delta_{1})$ and $P_{2}=E_{1}^{\perp}$
according to Theorem \ref{thm:main_thm}, straightening out $\Delta_{1}$
so that $\Delta_{1}=\mathrm{span}\{\partial_{\bar{x}_{1}^{+}}\}$
automatically leads to $E_{1}=\mathrm{span}\{\partial_{\bar{x}_{1}},\partial_{u}\}$
and $P_{2}=\mathrm{span}\{\mathrm{d}\bar{x}_{2}\}$. Furthermore,
due to $P_{2}^{+}+P_{1}=D_{0}^{\perp}$ as stated in Proposition \ref{prop:dual_D},
straightening out $D_{0}$ so that $D_{0}=\mathrm{span}\{\partial_{\bar{u}_{1}}\}$
leads to $P_{2}^{+}+P_{1}=\mathrm{span}\{\mathrm{d}\bar{x}_{1},\mathrm{d}\bar{x}_{2},\mathrm{d}\bar{u}_{2}\}$.
So, bringing the codistributions $P_{2}$ and $P_{2}^{+}+P_{1}$ into
the form $P_{2}=\mathrm{span}\{\mathrm{d}\bar{x}_{2}\}$ and $P_{2}^{+}+P_{1}=\mathrm{span}\{\mathrm{d}\bar{x}_{1},\mathrm{d}\bar{x}_{2},\mathrm{d}\bar{u}_{2}\}$,
gives rise to a state- and input transformation \eqref{eq:decompostion_state_transformation}
and \eqref{eq:decompostion_input_transformation}, which achieves
a decomposition (\ref{eq:basic_decomposition_flat}). In \cite{kolar2024dual},
the state transformation \eqref{eq:decompostion_state_transformation}
is obtained in exactly this way, but the input transformation \eqref{eq:decompostion_input_transformation}
is determined by simply ``normalizing'' $\mathrm{dim}(\bar{u}_{2})=\mathrm{rank}(\partial_{u}f_{2})$
equations of the subsystem \eqref{eq:decomposition_subsystem}.\\
In the following, we demonstrate that the input transformation derived
from the approach in \cite{kolar2024dual} is encompassed within the
set of possible input transformations obtained by straightening out
the sequence \eqref{eq:sequ_D_k} as described in \cite{KolarDiwoldSchoberl:2023}.
\begin{prop}
	\label{prop:trafo}Consider a system \eqref{eq:sysEq} with $\mathrm{rank}(\partial_{u}f)=m$
	which is transformed by a state transformation \eqref{eq:decompostion_state_transformation}
	into the form\begin{subequations}\label{eq:P2_straightened_out_sys}
		\begin{align}
			\bar{x}_{2}^{+} & =f_{2}(\bar{x}_{2},\bar{x}_{1},u)\label{eq:P2_straightened_out_subsys}\\
			\bar{x}_{1}^{+} & =f_{1}(\bar{x}_{2},\bar{x}_{1},u)\label{eq:P2_straightened_out_feedback}
		\end{align}
	\end{subequations} such that $\Delta_{1}=\mathrm{span}\{\partial_{\bar{x}_{1}^{+}}\}$
	(and $P_{2}=\mathrm{span}\{\mathrm{d}\bar{x}_{2}\}$). Performing
	an input transformation \eqref{eq:decompostion_input_transformation}
	with $\mathrm{dim}(\bar{u}_{2})=\mathrm{rank}(\partial_{u}f_{2})$
	such that \textup{$\mathrm{dim}(\bar{u}_{2})$} equations of the subsystem
	\eqref{eq:P2_straightened_out_subsys} are simplified to $\bar{x}_{2}^{i_{2},+}=\bar{u}_{2}^{i_{2}}$\footnote{This can be achieved by setting $\bar{u}_{2}^{i_{2}}=f_{2}^{i_{2}}(\bar{x}_{2},\bar{x}_{1},u)$
		and choosing $\bar{u}_{1}$ such that \eqref{eq:decompostion_input_transformation}
		is invertible w.r.t. $u$.} results in $D_{0}=\mathrm{span}\{\partial_{\bar{u}_{1}}\}$. Conversely,
	after an input transformation \eqref{eq:decompostion_input_transformation}
	which results in $D_{0}=\mathrm{span}\{\partial_{\bar{u}_{1}}\}$,
	there exists a further input transformation of the form $\tilde{u}_{2}=\Phi_{u_{2}}(\bar{x},\bar{u}_{2})$
	with $\mathrm{dim}(\tilde{u}_{2})=\mathrm{dim}(\bar{u}_{2})$ such
	that $\mathrm{dim}(\tilde{u}_{2})$ equations of the subsystem \eqref{eq:P2_straightened_out_subsys}
	take the form $\bar{x}_{2}^{i_{2},+}=\tilde{u}_{2}^{i_{2}}$.
\end{prop}

\begin{proof}
	In \cite{kolar2024dual}, it was shown in the proof of Theorem 15
	that by performing an input transformation \eqref{eq:decompostion_input_transformation}
	with $\mathrm{dim}(\bar{u}_{2})=\mathrm{rank}(\partial_{u}f_{2})$
	such that $\mathrm{dim}(\bar{u}_{2})$ equations of the subsystem
	\eqref{eq:P2_straightened_out_subsys} are simplified to $\bar{x}_{2}^{i_{2},+}=\bar{u}_{2}^{i_{2}}$
	the system (\ref{eq:P2_straightened_out_sys}) is transformed into
	the triangular form (\ref{eq:basic_decomposition_flat}). Since the
	Jacobian matrix $\partial_{\bar{u}_{2}}f_{2}$ has full rank, $D_{0}\not\subset\mathrm{span}\{\partial_{\bar{u}_{1}}\}$
	would immediately lead to $f_{*}(D_{0})\not\subset\mathrm{span}\{\partial_{\bar{x}_{1}^{+}}\}$.
	However, this is a contradiction to $\Delta_{1}=f_{*}(D_{0})=\mathrm{span}\{\partial_{\bar{x}_{1}^{+}}\}$.
	With $\mathrm{dim}(\bar{u}_{1})=\mathrm{dim}(\bar{x}_{1})$ we get
	$D_{0}=\mathrm{span}\{\partial_{\bar{u}_{1}}\}$.\\
	Conversely, as shown in \cite{KolarDiwoldSchoberl:2023}, an input
	transformation obtained by straightening out $D_{0}$ transforms the
	system (\ref{eq:P2_straightened_out_sys}) into the triangular form
	(\ref{eq:basic_decomposition_flat}). Because of the assumption that
	the Jacobian matrix $\partial_{u}f$ has full rank and the triangular
	structure of the system, it immediately follows that also the Jacobian
	matrix $\partial_{\bar{u}_{2}}f_{2}$ must have full rank. Hence,
	we can simplify $\mathrm{dim}(\bar{u}_{2})$ equations of the subsystem
	\eqref{eq:P2_straightened_out_subsys} to the form $\bar{x}_{2}^{i_{2},+}=\tilde{u}_{2}^{i_{2}}$
	by introducing new inputs $\tilde{u}_{2}^{i_{2}}=f_{2}^{i_{2}}(\bar{x}_{2},\bar{x}_{1},\bar{u}_{2})$.
\end{proof}
After such a decomposition step, the new inputs $(\bar{x}_{1},\bar{u}_{2})$
of the subsystem \eqref{eq:P2_straightened_out_subsys} may include
redundant inputs with $\partial_{(\bar{x}_{1},\bar{u}_{2})}f_{2}<m$.
However, as noted in Remark 3 of \cite{KolarSchoberlDiwold:2019},
the redundant inputs of the subsystem \eqref{eq:P2_straightened_out_subsys}
appear among the variables $\bar{x}_{1}$ and can therefore be eliminated
by an additional input transformation. They are also candidates for
components of a possible flat output of the sytem \eqref{eq:sysEq}.
As a result, Proposition \ref{prop:trafo} applies at every decomposition
step, enabling repeated system decompositions as in \cite{KolarDiwoldSchoberl:2023}.

\section{Example\label{sec:example}}

In this section, we illustrate the duality of the two forward-flatness
tests by an academic example.

\subsection{Academic Example}

Consider the system
\begin{equation}
	\begin{aligned}x^{1,+} & =\tfrac{x^{2}+x^{3}+3x^{4}}{u^{1}+2u^{2}+1}\\
		x^{2,+} & =x^{1}\left(x^{3}+1\right)\left(u^{1}+2u^{2}-3\right)+x^{4}-3u^{2}\\
		x^{3,+} & =u^{1}+2u^{2}\\
		x^{4,+} & =x^{1}\left(x^{3}+1\right)+u^{2}
	\end{aligned}
	\label{eq:sys_ex}
\end{equation}
with $n=4$, $m=2$ and the equilibrium $\left(x_{0},u_{0}\right)=(0,0)$,
which is not static feedback linearizable. To illustrate the duality
of the two tests from Subsection \ref{subsec:Geometric} and \ref{subsec:Geometric-Dual},
with the help of adapted coordinates
\begin{equation}
	\theta^{i}=f^{i}(x,u)\,,\:\xi^{1}=x^{1}\,,\:\xi^{2}=x^{3}\label{eq:adapt_coord_ex}
\end{equation}
with $i=1,\dots,n$ and the results from Section \ref{sec:Duality},
we show that $P_{k}=E_{k-1}^{\perp}$ applies for $k=1,\ldots,\bar{k}$
on $\mathcal{X}\times\mathcal{U}$.

\subsubsection{Distribution-based Test\label{subsec:Distribution-based-Test}}

Using Algorithm \ref{alg:distribution}, the sequence \eqref{eq:E_sequ}
results in
\begin{align}
	E_{0} & =\mathrm{span}\left\{ \partial_{u^{1}},\partial_{u^{2}}\right\} \nonumber \\
	E_{1} & =\mathrm{span}\left\{ -3\partial_{x^{2}}+\partial_{x^{4}},\partial_{u^{1}},\partial_{u^{2}}\right\} \label{eq:E_sequ_ex}\\
	E_{2} & =\mathrm{span}\left\{ \partial_{x^{1}}-\tfrac{x^{3}+1}{x^{1}}\partial_{x^{3}},\partial_{x^{2}},\partial_{x^{4}},\partial_{u^{1}},\partial_{u^{2}}\right\} \nonumber \\
	E_{3} & =\mathrm{span}\left\{ \partial_{x^{1}},\partial_{x^{2}},\partial_{x^{3}},\partial_{x^{4}},\partial_{u^{1}},\partial_{u^{2}}\right\} \,.\nonumber 
\end{align}
According to Theorem \ref{thm:condition}, the system \eqref{eq:sys_ex}
is forward-flat with $\mathrm{dim}(E_{3})=n+m$. Additionally, the
sequence \eqref{eq:sequ_D_k}, consisting of the projectable subdistributions
$D_{k-1}\subset E_{k-1}$ with $k=1,\ldots,\bar{k}-1$ is given by
\begin{equation}
	D_{0}=\mathrm{span}\left\{ -2\partial_{u^{1}}+\partial_{u^{2}}\right\} \,,\:D_{1}=E_{1}\,,\:D_{2}=E_{2}\label{eq:D_sequ_ex}
\end{equation}

\subsubsection{Codistribution-based Test\label{subsec:Codistribution-based-Test}}

Analogously to Subsection \ref{subsec:Distribution-based-Test},
the sequence \eqref{eq:P_sequ}, which is calculated according to
Algorithm \ref{alg:dual}, follows as
\begin{align}
	P_{1} & =\mathrm{span}\left\{ \mathrm{d}x^{1},\mathrm{d}x^{2},\mathrm{d}x^{3},\mathrm{d}x^{4}\right\} \nonumber \\
	P_{2} & =\mathrm{span}\left\{ \mathrm{d}x^{1},\mathrm{d}x^{3},\mathrm{d}x^{2}+3\mathrm{d}x^{4}\right\} \label{eq:P_sequ_ex}\\
	P_{3} & =\mathrm{span}\left\{ \tfrac{x^{3}+1}{x^{1}}\mathrm{d}x^{1}+\mathrm{d}x^{3}\right\} \nonumber \\
	P_{4} & =0\,.\nonumber 
\end{align}
Thus, the system \eqref{eq:sys_ex} is forward-flat according to
Theorem \ref{thm:dual_condition} as well. Finally, the sequence \eqref{eq:D_perp}
can be stated as
\begin{equation}
	\begin{aligned}P_{2}^{+}+P_{1} & =\mathrm{span}\left\{ \mathrm{d}x^{1},\ldots,\mathrm{d}x^{4},\mathrm{d}u^{1}+2\mathrm{d}u^{2}\right\} \\
		P_{3}^{+}+P_{2} & =P_{2}\\
		P_{4}^{+}+P_{3} & =P_{3}\,.
	\end{aligned}
	\label{eq:D_bot_sequ_ex}
\end{equation}

\subsubsection{Comparison}

From the sequence \eqref{eq:D_sequ_ex}, it can be concluded that
only in the first step with $k=1$, the distribution $E_{0}$, which
in adapted coordinates \eqref{eq:adapt_coord_ex} can be depicted
as 
\[
\begin{aligned}E_{0}= & \mathrm{span}\left\{ \mathrm{\partial_{\theta^{1}}}-\tfrac{\theta^{3}+1}{\theta^{1}}\partial_{\theta^{3}}-\tfrac{\xi^{1}(\xi^{2}+1)(\theta^{3}+1)}{3\theta^{1}}\partial_{\theta^{4}},\right.\\
	& \left.\partial_{\theta^{2}}-\tfrac{1}{3}\partial_{\theta^{4}}\right\} 
\end{aligned}
\]
with normalized base elements of the form \eqref{eq:v_bar_d_norm},
is not projectable. Analogously, as can be inferred from \eqref{eq:D_bot_sequ_ex},
the intersection $P_{1}\cap\mathrm{span}\{\mathrm{d}\theta\}$ with
the codistribution 
\[
\begin{aligned}P_{1}= & \mathrm{span}\left\{ \mathrm{\tfrac{\theta^{3}+1}{\theta^{1}}}\mathrm{d}\theta^{1}+\mathrm{d}\theta^{3},\tfrac{\xi^{1}(\xi^{2}+1)(\theta^{3}+1)}{3\theta^{1}}\mathrm{d}\theta^{1}\right.\\
	& \left.+\tfrac{1}{3}\mathrm{d}\theta^{2}+\mathrm{d}\theta^{4},\mathrm{d}\xi^{1},\mathrm{d}\xi^{2}\right\} 
\end{aligned}
\]
spanned by base elements of the form \eqref{eq:omega_theta_norm}-\eqref{eq:omega_xi_norm},
is not yet invariant w.r.t. $\mathrm{span}\{\partial_{\xi}\}$. All
further steps with $k=2,\ldots,\bar{k}$ are trivial, as the distributions
$E_{k}$ are projectable and the intersections $P_{k}\cap\mathrm{span}\{\mathrm{d}\theta\}$
are already $\xi$-invariant. With the matrix
\[
\hat{M}=\begin{bmatrix}-\tfrac{(\xi^{2}+1)(\theta^{3}+1)}{3\theta^{1}} & 0\\
	-\tfrac{\xi^{1}(\theta^{3}+1)}{3\theta^{1}} & 0
\end{bmatrix}\,,
\]
which consists of the linearly independent rows of \eqref{eq:matrix_M},
the largest projectable subdistribution $D_{0}=\mathrm{span}\{-3\partial_{\theta^{2}}+\partial_{\theta^{4}}\}\subset E_{0}$
($D_{0}=\mathrm{span}\{-2\partial_{u^{1}}+\partial_{u^{2}}\}$ in
coordinates $(x,u)$) follows immediately with the application of
Proposition \ref{prop:D_bar}. Furthermore, according to \eqref{eq:added_1forms}
from Proposition \ref{prop:P_hat}, the extension of $P_{1}\cap\mathrm{span}\{\mathrm{d}\theta\}$
by the 1-form $\rho^{1}=\mathrm{d}\theta^{1}$ results in the invariant
codistribution $P_{2}^{+}$ with a $\xi$-independent basis given
by $P_{2}^{+}=\mathrm{span}\left\{ \mathrm{d}\theta^{1},\mathrm{d}\theta^{3},\mathrm{d}\theta^{2}+3\mathrm{d}\theta^{4}\right\} $.
With $E_{1}=\pi_{*}^{-1}(f_{*}(D_{0}))$ and $P_{2}=\delta^{-1}(P_{2}^{+})$,
the corresponding elements from \eqref{eq:E_sequ_ex} and \eqref{eq:P_sequ_ex}
follow immediately. When comparing \eqref{eq:E_sequ_ex} and \eqref{eq:P_sequ_ex},
the interior product $v\rfloor\omega$ yields $0$ for all base elements
$\omega\in P_{k}$ and $v\in E_{k-1}$ and $\mathrm{dim}(E_{k-1})+\mathrm{dim}(P_{k})=n+m$.
So, we have shown that the relation $P_{k}=E_{k-1}^{\perp}$ holds,
which is consistent with Theorem \ref{thm:main_thm}. Furthermore,
\eqref{eq:D_sequ_ex} and \eqref{eq:D_bot_sequ_ex} also satisfy the
conditions $D_{k-1}\rfloor(P_{k+1}^{+}+P_{k})=0$ and $\mathrm{dim}(D_{k-1})+\mathrm{dim}(P_{k+1}^{+}+P_{k})=n+m$,
which shows the validity of Proposition \eqref{prop:dual_D}.

\section{Conclusion}

We have shown that for the Algorithms \ref{alg:distribution} and
\ref{alg:dual}, starting with the distribution $E_{0}=\mathrm{span}\{\partial_{u}\}$
and the codistribution $P_{1}=\mathrm{span}\{\mathrm{d}x\}$, the
relation $P_{k}=E_{k-1}^{\perp}$ holds for every step $k=1,\ldots,\bar{k}$.
That is, the sequence of integrable codistributions \eqref{eq:P_sequ}
annihilates the sequence of involutive distributions \eqref{eq:E_sequ}
on the manifold $\mathcal{X}\times\mathcal{U}$. Consequently, the
distribution-based geometric test from Subsection \ref{subsec:Geometric}
and the codistribution-based geometric test from Subsection \ref{subsec:Geometric-Dual}
are indeed dual. Additionally, in Proposition \ref{prop:D_bar}, we
established an efficient method for computing the largest projectable
subdistribution $\bar{D}$ of a distribution $D$, which is an essential
step in Algorithm \ref{alg:distribution}. It was also shown that
while the state transformations \eqref{eq:decompostion_state_transformation}
required for repeated system decompositions of the form (\ref{eq:basic_decomposition_flat})
as proposed in the papers \cite{KolarDiwoldSchoberl:2023} and \cite{kolar2024dual}
do align, the same is not true for the input transformations \eqref{eq:decompostion_input_transformation}.
We have shown that the input transformation derived from the method
in \cite{kolar2024dual}, which only requires normalizing a certain
number of equations, automatically straightens out the distributions
$D_{k}$. It is therefore a special case of the possible input transformations
derived from the method in \cite{KolarDiwoldSchoberl:2023}.

\begin{ack}                               
This research was funded in whole, or in part, by the Austrian Science Fund (FWF) P36473. For the purpose of open access, the author has applied a CC BY public copyright licence to any Author Accepted Manuscript version arising from this submission.
\end{ack}

\bibliographystyle{plain}        
\bibliography{Bibliography}



\end{document}